\documentclass[a4paper, reqno, english,11pt]{smfart}

\usepackage{latexsym, amsfonts,amssymb,amsmath,amsthm}

\newtheorem{prop}{Proposition}

\begin{document}

\newtheorem{conjecture}{Conjecture}
\newtheorem{theorem}{Theorem}
\newtheorem{lemma}{Lemma}
\newtheorem{corollary}{Corollary}
\numberwithin{equation}{section}
\newcommand{\dif}{\mathrm{d}}
\newcommand{\Res}{\mathrm{Res}}
\newcommand{\intz}{\mathbb{Z}}
\newcommand{\ratq}{\mathbb{Q}}
\newcommand{\natn}{\mathbb{N}}
\newcommand{\comc}{\mathbb{C}}
\newcommand{\rear}{\mathbb{R}}
\newcommand{\prip}{\mathbb{P}}
\newcommand{\uph}{\mathbb{H}}
\newcommand{\fief}{\mathbb{F}}
\newcommand{\majorarc}{\mathfrak{M}}
\newcommand{\minorarc}{\mathfrak{m}}
\newcommand{\sings}{\mathfrak{S}}
\newcommand{\dds}{\frac{\dif}{\dif s}}

\renewcommand{\d}{\mathrm{d}}
\renewcommand{\phi}{\varphi}
\newcommand{\PP}{\mathbb{P}}
\renewcommand{\AA}{\mathbb{A}}
\newcommand{\FF}{\mathbb{F}}
\newcommand{\ZZ}{\mathbb{Z}}
\newcommand{\ZZp}{\mathbb{Z}_{\mathrm{prim}}}
\newcommand{\NN}{\mathbb{N}}
\newcommand{\QQ}{\mathbb{Q}}
\newcommand{\RR}{\mathbb{R}}
\newcommand{\CC}{\mathbb{C}}

\newcommand{\XX}{\boldsymbol{X}}
\newcommand{\YY}{\boldsymbol{Y}}
\newcommand{\UU}{\boldsymbol{U}}

\renewcommand{\leq}{\leqslant}
\renewcommand{\le}{\leqslant}
\renewcommand{\geq}{\geqslant}
\renewcommand{\ge}{\geqslant}
\renewcommand{\bar}{\overline}

\newcommand{\ma}{\mathbf}

\newcommand{\m}{\mathbf{m}}
\newcommand{\M}{\mathbf{M}}
\newcommand{\y}{\mathbf{y}}
\renewcommand{\c}{\mathbf{c}}
\newcommand{\z}{\mathbf{z}}
\renewcommand{\b}{\mathbf{b}}
\newcommand{\la}{\lambda}
\newcommand{\al}{\alpha}

\newcommand{\bet}{\boldsymbol{\eta}}
\newcommand{\bnu}{\boldsymbol{\nu}}
\newcommand{\bla}{\boldsymbol{\lambda}}
\newcommand{\bal}{\boldsymbol{\alpha}}
\newcommand{\x}{\mathbf{x}}
\newcommand{\ve}{\varepsilon}

\theoremstyle{definition}
\newtheorem*{ack}{Acknowledgements}
\newtheorem*{notat}{Notation}

\title{Averages of shifted convolutions of $d_3(n)$}

\author{S.\ Baier}
\address{School of Mathematics\\
University of Bristol\\ Bristol\\ BS8 1TW\\ United Kingdom}
\email{stephan.baier@bristol.ac.uk}
\author{T.\ D.\ Browning}
\address{School of Mathematics\\
University of Bristol\\ Bristol\\ BS8 1TW\\ United Kingdom}
\email{t.d.browning@bristol.ac.uk}
\author{G.\ Marasingha}
\address{School of Mathematics\\
University of Bristol\\ Bristol\\ BS8 1TW\\ United Kingdom}
\email{gihan.marasingha@bristol.ac.uk}
\author{L.\ Zhao}
\address{Division of Math. Sci. \\
School of Phys. \& Math. Sci. \\
Nanyang Technological University \\
Singapore 637371}
\email{lzhao@pmail.ntu.edu.sg}

\date{\today}

\begin{abstract}
We investigate the first and second moments of shifted convolutions of
the generalised divisor function $d_3(n)$. 
\end{abstract}

\subjclass{11N37 (11M06, 11P55)}

\maketitle 
\tableofcontents

\section{Introduction}

For any positive integer $k$ let $d_k(n)$ denote the generalised divisor function,
defined to be the Dirichlet coefficients of $\zeta(s)^k$ in the
half-plane $\Re(s)>1$. The study of shifted convolution sums
$$
D_k(N, h) := \sum_{N<n \le 2N} d_k(n) d_k(n+h)
$$
is of central importance in the analytic theory of
numbers.  The case $k=1$ is trivial and for $k=2$ we have known since
work of Ingham \cite{ing} that
$$
D_2(N,h) \sim \frac{6}{\pi^2}\sigma_{-1}(h)N \log^2 N
$$
as $N \rightarrow \infty$, for given $h \in \NN$, where
$\sigma_{-1}(h):=\sum_{j\mid h} j^{-1}$. 
Several authors have since 
revisited this problem, achieving asymptotic formulae 
with $h$ in an increasingly
large range compared to $N$. The best results in the literature are
due to Duke, Friedlander and Iwaniec \cite{dfi} and to 
Meurman \cite{meurman-d2}.

In general it is expected that 
$D_k(N, h)$ should be asymptotic to $c_{k,h} N \log^{2k-2}N$, for a
suitable constant $c_{k,h}>0$, uniformly for $h$ in some range.
However such a description has not yet been forthcoming 
for any $k\geq 3$, even when $h$ is fixed. 
One motivation for studying the sums $D_k(N,h)$ is the deep connection
that they enjoy with 
the asymptotic behaviour of  moments
$$
I_k(T):=\int\limits_0^T \left|\zeta\left(\frac{1}{2}+it\right)\right|^{2k} \d
t, 
$$
as $T\rightarrow \infty$.  It is commonly believed that
$
I_k(T)\sim c_k T(\log T)^{k^2},
$ 
as $T\rightarrow \infty$, for a
suitable constant $c_k>0$. Keating and Snaith
\cite{ks} have produced a conjectural interpretation of $c_k$
using random matrix theory for Gaussian unitary ensembles. 
Just as for the sums $D_k(N,h)$, we have 
only succeeded in producing an asymptotic formula 
for $I_k(T)$ when $k=1$
(see Hardy and Littlewood \cite{hl}) or $k=2$ (see Ingham \cite{ingham}). 
The relationship between moments of the Riemann zeta function and 
the shifted convolution sums $D_k(N,h)$ has been explored extensively by 
Ivi\'c \cite{ivic-art, ivic-art'}, 
and more recently by Conrey and Gonek
\cite{conrey-gonek}.  

Fixing attention on the case $k=3$, in which setting we write
$D(N,h)=D_3(N,h)$, our aim in this paper is to lend some theoretical
support in favour of its expected asymptotic behaviour.
If $\varphi(n)$ denotes the Euler totient function, then we set
$$
H(s,q):= \sum\limits_{d\mid q} \frac{\mu(d)}{\varphi(d)}\cdot d^s\cdot
G_{q/d,d}(s), 
$$
with 
\begin{equation} \label{Gkddef}
G_{k,d}(s):=\sum\limits_{e\mid d} \frac{\mu(e)}{e^s}\cdot g(s,ek)
\end{equation}
and 
$$
g(s,q):=\prod\limits_{p\mid q} \left(\left(1-p^{-s}\right)^3 \sum\limits_{j=0}^{\infty} \frac{d_3\left(p^{j+v_p(q)}\right)}{p^{js}}\right).
$$
Here and after, $v_p(q)$ denotes the $p$-adic valuation of $q$.  Next we define
\begin{equation} \label{Pdefini}
\begin{split}
P(x,q)&:=
\frac{1}{2\pi i} \int\limits_{|s|=1/8} \zeta^3(s+1) H(s+1,q)
\left(\frac{x}{q}\right)^s \dif s\\ 
&=
\Res_{s=0} \zeta^3(s+1) H(s+1,q) \left(\frac{x}{q}\right)^s,
\end{split}
\end{equation}
by the residue theorem.  Let 
$c_q(h)=\sum_{d\mid h,q} d\mu(q/d)$ be the Ramanujan sum and 
let $\ve>0$.  Then the work of 
Conrey and Gonek \cite[Eq.\ (30) and 
Conjecture~3]{conrey-gonek} predicts that
\begin{equation} \label{congonrew}
D(N,h)=\int\limits_{N}^{2N} \mathfrak{S}(x,h) \dif x+O(N^{1/2+\ve}),
\end{equation}
uniformly for $1\leq h\leq x^{1/2}$, where
\begin{equation}
  \label{eq:dijon}
  \mathfrak{S}(x,h):=\sum\limits_{q=1}^{\infty} \frac{c_q(h)}{q^2}\cdot P(x,q)^2. 
\end{equation}

Let
$$
\Delta(N,h):=D(N,h)-\int\limits_{N}^{2N} \mathfrak{S}(x,h) \dif x.
$$
We will lend support to \eqref{congonrew} by considering 
both first and second moments of 
$\Delta(N,h)$, as $h$ varies over some range that is small compared to
$N$. 
Beginning with the former, we will
establish the following result.

\begin{theorem} \label{mainres1} Assume
  that $1\le H\le N$. Then 
$$
\sum\limits_{h\le H} \Delta(N,h)\ll
\left(H^2+H^{1/2}N^{13/12} \right)N^{\varepsilon}.
$$
\end{theorem}

The exponents appearing in this estimate can be improved slightly  for
certain ranges of $H$. We shall not pursue this here however. 
For $N$ in the range $N^{1/6+\varepsilon}\le H\le N^{1-\varepsilon}$,
Theorem \ref{mainres1} gives an asymptotic formula for the average
\begin{equation} \label{Gaverage}
G(N,H):=\sum\limits_{h\le H} D(N,h).
\end{equation} 
It is interesting to relate Theorem 1 to work of Ivi\'c
\cite[Lemma 6]{ivic-art'} who deduces the upper bound
$$
I_3(T)\ll T^{1+\ve}+T^{(\al+3\beta-1)/2+\ve}
$$
for the sixth moment of the Riemann zeta function on the critical line,
where $\alpha,\beta\in [0,1]$ are constants such that
$\alpha+\beta\geq 1$
and an asymptotic formula of the shape
$$
\sum_{h \leq H} \Delta(N,h)
\ll H^\alpha N^{\beta+\varepsilon}
$$
is valid for $1\leq H\leq N^{1/3}$.
Theorem 1 affords the choices $\alpha=1/2$ and $\beta=13/12$, which
yields $I_3(T)\ll T^{11/8+\ve}$.  Unfortunately this does 
not give any improvement over the well-known bound for $I_3(T)$ with 
exponent $5/4+\ve$.

Turning to second moments we will establish the following result. 

\begin{theorem} \label{mainres2}
Assume
that $N^{1/3+\varepsilon}\le H\le N^{1-\varepsilon}$. Then
there exists $\delta>0$ such that 
$$
\sum\limits_{h\le H} \left| \Delta(N,h) \right|^2 \ll HN^{2-\delta}. 
$$
\end{theorem}

It follows from Theorem \ref{mainres2} 
that the expected asymptotic formula 
$$
D(N,h)\sim \int\limits_{N}^{2N} \mathfrak{S}(x,h) \dif x
$$
holds for almost all $h\le H$ if $N^{1/3+\varepsilon}\le H\le N^{1-\varepsilon}$.
Our proof of Theorem \ref{mainres2} is based on Mikawa's investigation
\cite{mikawa} of twin primes. Here the Hardy--Littlewood circle method
is adapted to study the second moment of the 
analogous shifted convolution sum in which $d_3(n)$
is replaced by the von Mangoldt function $\Lambda(n)$.  Our proof of
Theorem \ref{mainres1} is simpler, being based on Perron's
formula and a bound for the sixth moment of the Riemann zeta
function. 

\begin{notat}
Our work will involve small positive parameters $\ve$ and 
$\delta$, $\delta_1, \delta_2, \ldots$. The
value of $\varepsilon$ will be allowed to vary from line to line, and 
$\delta$, $\delta_1, \delta_2, \ldots$ may depend on
$\ve$.  All of the implied constants in our work are permitted to
depend at most on these parameters.
\end{notat}

\begin{ack}
  While working on this paper the first two authors were supported by
  EPSRC grant number \texttt{EP/E053262/1}, the third author was
  sponsored by the {\em Heilbronn Institute for Mathematical Research},
  and the fourth author by an AcRF Tier I
  grant.  Moreover, the first and fourth authors wish to thank the
  {\it Mathematisches Forschungsinstitut Oberwolfach}, where some
  parts of this paper were prepared during a ``Research in Pairs''
  programme. 
Finally the authors are grateful to Aleksandar Ivi\'c and Jie Wu for
  a number of useful comments on an earlier version of the manuscript.
\end{ack}

\section{Estimation of $G(N,H)$}\label{s:G}

The following two sections deal with the proof of Theorem
\ref{mainres1}. To this end, we evaluate separately the averages
$G(N,H)$, defined in \eqref{Gaverage}, and  
\begin{equation} \label{Fdefin}
F(N,H):=\sum\limits_{h\le H} \int\limits_{N}^{2N} \mathfrak{S}(x,h) \dif x.
\end{equation}
We begin with the more complicated evaluation of $G(N,H)$. 
Changing the order of summation, we get
\begin{equation} \label{change} 
G(N,H) = \sum_{N<n \le 2N } d_3(n) \sum_{h \leq H} d_3(n+h). 
\end{equation}
Using Perron's formula, the inner sum in \eqref{change} can be expressed in the form
\begin{equation} \label{inner}
 \sum_{h \leq H} d_3(n+h) = \frac{1}{2\pi i} \int\limits_{c-iT}^{c+iT} \zeta^3(s) \left( (n+H)^s -n^s \right) \frac{\dif s}{s} + O \left( \frac{N^{1+\varepsilon}}{T} \right), 
 \end{equation}
where $c = 1 + (\log N)^{-1}$ and $2\le T \le N$.
Shifting the line of integration and using the residue theorem, we see
that  the integral is 
\begin{equation} \label{moveint}
\begin{split}
\Res_{s=1} \zeta^3(s)\frac{(n+H)^s-n^s}{s} 
+ \frac{1}{2 \pi i} \left( \ \int\limits_{\mathcal{P}_1} +  \int\limits_{\mathcal{P}_2} + \int\limits_{\sigma-iT}^{\sigma + iT} \ \right) \zeta^3(s) \left( (n+H)^s -n^s \right) \frac{\dif s}{s} ,
\end{split}
\end{equation}
where $1/2 < \sigma <1$ is a parameter to be
fixed later, $\mathcal{P}_1$ is the line segment connecting $c-iT$ and $\sigma
- iT$, and $\mathcal{P}_2$ is the line segment connecting $\sigma + iT$ and
$c+iT$ . 

For $1/2\leq |\alpha|\leq 1$ and 
$|t| \geq 1$,  
Weyl's subconvexity bound is 
$
\zeta( \alpha +it ) \ll |t|^{(1-\alpha)/3+\varepsilon}.
$
Moreover, $\zeta(\alpha\pm iT)\ll \log T$ uniformly in $1\le \alpha
\le c$. Hence, for $i=1,2$, 
the integrals over $\mathcal{P}_i$ in \eqref{moveint} are bounded by 
\begin{equation} \label{horzint}
 \int\limits_{\mathcal{P}_i} \zeta^3(s) \left( (n+H)^s -n^s \right)
 \frac{\dif s}{s} \ll 
\frac{N^{\varepsilon}}{T} \int\limits_{\sigma}^1 T^{1-\alpha} N^{\alpha} \dif \alpha \ll \frac{N^{1+\varepsilon}}{T}, 
 \end{equation}
where we take into account that $2 \leq T \le N$ and $N<n \le 2N$. 

Combining this with \eqref{change}, \eqref{inner} and 
\eqref{moveint}, we therefore obtain 
\begin{equation} \label{Gest1}
G(N,H) =    M(N,H) + E(N,H) +
O\left(\frac{N^{2+\varepsilon}}{T}\right), 
\end{equation}
where
$$
M(N,H) := \sum_{N<n\le 2N} d_3(n) \Res_{s=1} \zeta^3(s)
\frac{(n+H)^s-n^s}{s} 
$$
and
\begin{equation} \label{error}
E(N,H) := \frac{1}{2\pi i} \int\limits_{\sigma-iT}^{\sigma + iT} \zeta^3(s) \sum_{N<n\le 2N} d_3(n) ((n+H)^{s}-n^s) \frac{\dif s}{s}.
\end{equation}

We proceed by writing 
\[
M(N,H) = \sum_{N < n  \le 2 N} d_3(n) g(n),
\]
with 
\[
g(x) := \Res_{s=1} \zeta^3(s) \frac{(x+H)^s-x^s}{s}.
\]
We note that $g(x)\ll H x^\varepsilon$ and 
$g'(x)\ll H x^{\varepsilon-1}$.
Thus partial summation yields
\[
M(N,H) = g(2N)\sum_{N<n\le 2N} d_3(n) - \int\limits_N^{2N} g'(t) \sum_{N < n \le t}
d_3(n) \dif t.
\]
The classical work of Voronoi \cite[Theorem 12.2]{titch} yields
\[
\sum_{n \le t} d_3(n) = \Res_{s=1} \zeta^3(s) \frac{t^s}{s}  +O(t^{1/2+\varepsilon}).
\]
From these results we deduce that 
\begin{align*}
  M(N,H) = ~& 
g(2N)
\cdot \left(\Res_{s=1} \zeta^3(s) \frac{(2N)^s-N^s}{s} \right)  
   - \int\limits_N^{2N} g'(t)
 \cdot \left(\Res_{s=1} \zeta^3(s) \frac{t^s-N^s}{s} \right)
  \dif t \\ 
&+ O ( HN^{1/2+\varepsilon}).
  \end{align*}
Integration by parts now reveals that
\begin{equation} \label{nr1}
  M(N,H)= \int\limits_N^{2N} 
g(t)
\cdot \left(\frac{\dif}{\dif t} \Res_{s=1} \zeta^3(s) \frac{t^s-N^s}{s} \right)
  \dif t +  O( HN^{1/2+\varepsilon}). 
\end{equation}
Employing the Taylor series expansion 
$$
\frac{(t+H)^s-t^s}{s}=Ht^{s-1}+\frac{H^2}{2}(s-1)t^{s-2}+\cdots$$
and the Laurent series expansion
for $\zeta^3(s)$ about $s=1$, we obtain
$$
\Res_{s=1} \zeta^3(s) \frac{(t+H)^s-t^s}{s} = H \Res_{s=1} \zeta^3(s)t^{s-1}+O\left(\frac{H^2}{t}\right),
$$
where we keep in mind that $H\le N$. Moreover,
$$
\frac{\dif}{\dif t} \Res_{s=1} \zeta^3(s) \frac{t^s-N^s}{s} =\Res_{s=1} \zeta^3(s)t^{s-1} \ll t^\varepsilon.
$$
Putting these facts 
together in \eqref{nr1},  we obtain
\begin{equation} \label{MNH}
M(N,H)=H\int\limits_{N}^{2N} \left(\Res_{s=1} \zeta^3(s)t^{s-1} \right)^2 \dif t+O(H^2N^{\varepsilon}+HN^{1/2+\varepsilon}).
\end{equation}

Our next task is to estimate $E(N,H)$ in \eqref{error}. 
Applying partial summation to the sum over $n$, we
see that
\begin{align*}
\sum_{N<n\le 2N} d_3(n) ((n+H)^{s}-n^s) =~& \sum\limits_{N<n\le 2N}
\left(\left(1+\frac{H}{n}\right)^s-1\right) d_3(n)n^s  \\ 
=~&\left(\left(1+\frac{H}{2N}\right)^s-1\right)\sum\limits_{N<n\le 2N}
d_3(n)n^s \\ 
 &+ sH\int\limits_{N}^{2N} \left(1+\frac{H}{x}\right)^{s-1}
 \left(\sum\limits_{N<n\le x} d_3(n)n^s\right) \frac{\dif
   x}{x^2}.
 \end{align*}
It follows that
\begin{equation} \label{errorsplit}
E(N,H)=E_1(N,H)+E_2(N,H),
\end{equation}
where
\begin{align*}
E_1(N,H)&:= \frac{1}{2\pi i} \int\limits_{\sigma-iT}^{\sigma + iT} \zeta^3(s) \left(\left(1+\frac{H}{2N}\right)^s-1\right) \left(\sum\limits_{N<n\le 2N} d_3(n)n^s\right) \frac{\dif s}{s}\\
&= \frac{1}{4\pi iN} \int\limits_{0}^{H}
\int\limits_{\sigma-iT}^{\sigma + iT} \zeta^3(s)
\left(1+\frac{\theta}{2N}\right)^{s-1} \left(\sum\limits_{N<n\le 2N}
  d_3(n)n^s\right) \dif s \dif \theta,
\end{align*}
and 
$$
E_2(N,H):= \frac{H}{2\pi i} \int\limits_{N}^{2N} \int\limits_{\sigma-iT}^{\sigma + iT} \zeta^3(s) \left(1+\frac{H}{x}\right)^{s-1} \left(\sum\limits_{N<n\le x} d_3(n)n^s\right) \frac{\dif s \dif x}{x^2}.
$$
For $i=1,2$ we may deduce that
\begin{equation} \label{firstbound}
E_i(N,H)\ll \frac{H}{N} \sup\limits_{N<x\le 2N}\  \int\limits_{-T}^{T} |\zeta(\sigma+it)|^3 \cdot \left|\sum\limits_{N<n\le x} d_3(n)n^{\sigma+it}\right| \dif t.
\end{equation}

Next, we transform the inner sum over $n$ in \eqref{firstbound} with a
further application of Perron's formula, obtaining
\begin{equation} \label{inner2}
 \sum_{N<n \leq x} d_3(n)n^s = \frac{1}{2\pi i} \int\limits_{c_1-2iT}^{c_1+2iT} \zeta^3(s_1-s) \left(x^{s_1}-N^{s_1}\right) \frac{\dif s_1}{s_1} + O \left( \frac{N^{1+\sigma+\varepsilon}}{T} \right), 
 \end{equation}
where $s=\sigma+it$ and $c_1= 1 + \sigma+ (\log N)^{-1}$.  We will
shift 
the line of integration and use the residue theorem, noting 
that we
cross the pole of the zeta function at 1 since $|t|\le T$. In this way
we see that the integral is
\begin{equation}
  \label{moveint2}
\Res_{s_1=1+s} 
\zeta^3(s_1-s)\cdot  \frac{x^{s_1}-N^{s_1}}{s_1}
+ \frac{1}{2 \pi i} \left( \ \int\limits_{\mathcal{P}_3} +
  \int\limits_{\mathcal{P}_4} + \int\limits_{2\sigma-2iT}^{2\sigma+2
    iT} \ \right) \zeta^3(s_1-s) \left(x^{s_1} -N^{s_1} \right)
\frac{\dif s_1}{s_1}, 
\end{equation}
where $\mathcal{P}_3$ is the line segment connecting $c_1 - 2iT$ to 
$2\sigma-2iT$, and $\mathcal{P}_4$ is the line segment connecting
$2\sigma+ 2iT$ to  $c_1+2iT$.  

In the same way as \eqref{horzint}, we see that 
$$
 \int\limits_{\mathcal{P}_i} \zeta^3(s_1-s) \left(x^{s_1} -N^{s_1}
 \right) \frac{\dif s_1}{s_1} \ll
 \frac{N^{1+\sigma+\varepsilon}}{T},$$ 
for $i=3,4$, where we take into account that $|t|\le T$.

From \eqref{firstbound} and  \eqref{inner2}, we deduce that
\begin{equation} \label{furthersplitting}
E_i(N,H)\ll A(N,H)+B(N,H), 
\end{equation}
for $i=1,2$,
where
$$
A(N,H):=HN^{2\sigma-1} \int\limits_{-T}^{T} \int\limits_{-2T}^{2T}
|\zeta(\sigma+it)|^3 |\zeta(\sigma+i(t_1-t))|^3 \frac{\dif
  t_1}{1+|t_1|} \dif t  
$$
and 
$$
B(N,H):=HN^{\sigma+\varepsilon} \int\limits_{-T}^{T}
|\zeta(\sigma+it)|^3 \frac{\dif t}{1+|t|}. 
$$
Here $A(N,H)$ bounds the contribution of the third integral on the
right-hand side of \eqref{moveint2}, and $B(N,H)$ bounds the
contributions  from the remaining terms.

Since $\sigma>1/2$, we have
$$
B(N,H)\ll HN^{\sigma+\varepsilon}
$$
by the familiar bound for the third moment of the Riemann zeta function. 
Next, using Cauchy--Schwarz, we obtain
$$
A(N,H)\ll HN^{2\sigma-1} \int\limits_{-2T}^{2T}  \left( \ \int\limits_{-T}^{T} |\zeta(\sigma+it)|^6 \dif t\right)^{1/2}  \left( \ \int\limits_{-T}^{T} |\zeta(\sigma+i(t_1-t))|^6 \dif t\right)^{1/2}  
\frac{\dif t_1}{1+|t_1|}.
$$
Now we choose $\sigma=7/12$. By \cite[Eq.\ (8.80)]{ivic}, we have the
expected bound for the sixth zeta moment on the line $\Re
s=7/12$. Hence
$$
A(N,H)\ll HN^{2\sigma-1}T^{1+\varepsilon}\ll HN^{1/6+\varepsilon}T.
$$
It therefore follows that 
$$
A(N,H)+B(N,H) \ll 
HN^{1/6+\varepsilon}T+ 
HN^{7/12+\varepsilon}.
$$
We will balance this bound with the estimate in 
\eqref{Gest1} by choosing 
$T = H^{-1/2} N^{11/12}$.
Combining this with \eqref{Gest1}, \eqref{MNH}, 
\eqref{errorsplit} and  \eqref{furthersplitting} 
we now get the final asymptotic formula 
\begin{equation} \label{Geval}
G(N,H)=H\int\limits_{N}^{2N} \left(\Res_{s=1} \zeta^3(s)t^{s-1}
\right)^2 \dif t+O(H^2 N^{\varepsilon}+
  H^{1/2}N^{13/12+\varepsilon} ).
\end{equation}
Here we have observed that 
$HN^{7/12}\leq H^{1/2}N^{13/12}$ for $H\leq N$.

\section{Estimation of $F(N,H)$}\label{s:F}
It remains to evaluate $F(N,H)$, defined in \eqref{Fdefin}, and to estimate the difference
\begin{equation} \label{difference}
\sum\limits_{h\le H} \Delta(N,h)=G(N,H)-F(N,H).
\end{equation}
We observe that
\begin{equation}\label{mathfraksplit}
\begin{split}
\sum\limits_{h\le H} \mathfrak{S}(x,h) &=\sum\limits_{h\le H}
\sum\limits_{q=1}^{\infty} \frac{c_q(h)}{q^2}\cdot P(x,q)^2\\ 
&= \sum\limits_{q\le H} \left(\sum\limits_{h\le H} c_q(h)\right) \cdot
\frac{P(x,q)^2}{q^2}+\sum\limits_{h\le H} \sum\limits_{q> H}
\frac{c_q(h)}{q^2} \cdot P(x,q)^2.
\end{split}\end{equation}
In section \ref{s:SS}, we shall show that $P(x,q)=P^*(x,q)$, where
$P^*(x,q)$ is defined as in \eqref{P3}. Applying 
\eqref{eq:P3-upper} we therefore obtain
$
P(x,q)\ll (qx)^{\varepsilon},
$
since $H\leq N$ and $x\leq 2N$.
Using this and the fact that $|c_q(h)|\le (q,h)$, we deduce that
\begin{align*}
\sum\limits_{h\le H} \sum\limits_{q> H} \frac{c_q(h)}{q^2} \cdot
P(x,q)^2 &\ll x^{\varepsilon} \sum\limits_{h\le H} \sum\limits_{q> H}
\frac{(q,h)}{q^{2-\varepsilon}}  \\ 
&\ll  x^{\varepsilon} \sum\limits_{h\le H} \sum\limits_{d\mid h}
\sum\limits_{\substack{q>H\\ d\mid q}} \frac{d}{q^{2-\varepsilon}}\\
&\ll  x^{\varepsilon} \sum\limits_{h\le H} \sum\limits_{d\mid h}
\left(\frac{H}{d}\right)^{-1+\varepsilon} \cdot
\frac{d}{d^{2-\varepsilon}}\\ 
&\ll (xH)^{\varepsilon}.  
\end{align*}
Next, we evaluate the first sum on the right-hand side of
\eqref{mathfraksplit}. An old result of Carmichael \cite{car}
asserts that 
$$
\sum_{h\leq q}c_q(h)=0,
$$
if $q>1$. Hence we see that
$$
\sum\limits_{h\le H} c_q(h) 
= \begin{cases} 
H+O(1), & \mbox{ if } q=1,\\
O(q^{1+\varepsilon}), & \mbox{ if } q>1. 
\end{cases}
$$
Putting all of this together, and using the definition of
$P(x,1)$ in \eqref{Pdefini}, we get 
\begin{align*} 
\sum\limits_{h\le H} \mathfrak{S}(x,h) &= H\cdot
P(x,1)^2+O\left((xH)^{\varepsilon}\right)\\ 
&= H \left(\Res_{s=1}
\zeta^3(s)x^{s-1}\right)^2+ O\left((xH)^{\varepsilon}\right).
\end{align*}
This implies that
$$
F(N,H)=\sum\limits_{h\le H} \int\limits_{N}^{2N} \mathfrak{S}(x,h)\dif
x =H \int\limits_{N}^{2N} \left(\Res_{s=1} \zeta^3(s)x^{s-1}\right)^2
\dif x+O(N^{1+\varepsilon}). 
$$
Combining this with \eqref{Geval} and \eqref{difference}, 
we therefore conclude the proof of Theorem \ref{mainres1}.

\section{Activation of the circle method}

Now we turn to the proof of Theorem \ref{mainres2}. We shall mimic
Mikawa's \cite{mikawa} treatment of the same problem for
$\Lambda(n)$ in place of $d_3(n)$. 
However, several of Mikawa's arguments need to be adjusted
to the present situation, and additional complications will occur. In
this section, we describe the general setup of the circle
method. 

We begin by observing that
\begin{equation} \label{Dint}
D(N,h)=\int\limits_{0}^{1} |S(\alpha)|^2 e(-\alpha h) \dif \alpha +
O(hN^{\varepsilon}),
\end{equation}
where 
$$
S(\alpha):=\sum\limits_{N<n\le 2N} d_3(n)e(n\alpha).
$$
Let
$
Q_1:=N^{\delta}$ and $Q:=N^{1/4},
$
for a small parameter $0<\delta<1/4$. 
We divide the integration 
into major and minor arcs as follows. The major arcs are defined as
$$
\mathfrak{M}:=\bigcup_{q\le Q_1} \bigcup_{\substack{1\le a\le q\\ (a,q)=1}} I_{q,a}, \quad
I_{q,a}:=\left[\frac{a}{q}-\frac{1}{qQ},\frac{a}{q}+\frac{1}{qQ}\right],
$$
and the minor arcs as
$$
\mathfrak{m}:=\left[Q^{-1},1+Q^{-1}\right]\setminus \mathfrak{M}.
$$
In the remainder of this paper we establish the following two results.
Taken  together with \eqref{Dint}, they imply Theorem \ref{mainres2}.

\begin{prop} \label{majorarcstheo} Let $0<\eta<1$ and let $\delta>0$
  be sufficiently small.
Then there exists $\delta_1>0$ depending on $\eta$ and $\delta$ such
that uniformly for $h\le N^{1-\eta}$, we have 
$$
\int\limits_{\mathfrak{M}} |S(\alpha)|^2 e(-\alpha h) \dif \alpha=\int\limits_{N}^{2N} \mathfrak{S}(x,h) \dif x + O(N^{1-\delta_1}).
$$
\end{prop}

\begin{prop} \label{minorarcstheo} Let $0<\eta<1/3$ 
and let $\delta>0$   be sufficiently small.
Then there exists $\delta_2>0$ depending on $\eta$ and $\delta$ such
that for $N^{1/3+\eta}\le H\le N^{1-\eta}$, we have 
$$
\sum\limits_{h\le H} \left|\int\limits_{\mathfrak{m}} |S(\alpha)|^2
  e(-\alpha h) \dif \alpha \right|^2\ll N^{2-\delta_2}. 
$$
\end{prop}

Before we can state all the lemmas needed in our method, we need to
introduce a certain Dirichlet series and compute a related
residue. Let $k,q\in \NN$ 
and let $\chi$ be a character modulo $q$. A function that will occur
frequently in our analysis is the Dirichlet series
\begin{equation}
  \label{eq:series}
  F_k(\chi,s):=\sum\limits_{n=1}^{\infty} \frac{\chi(n)d_3(nk)}{n^s},
\end{equation}
initially defined for $\Re(s)>1$. In the following, we convert this series into an Euler product and show that it can be meromorphically continued to the half plane $\Re(s)>0$, with a possible pole at $s=1$, depending on whether the character $\chi$ is principal or not.  

To start with, let $\Re(s)>1$. By $\mathcal{A}_k$, we denote the set of integers whose prime divisors all divide $k$. Obviously, we can factor $F_k(\chi,s)$ in the form
\begin{equation} \label{factorization}
F_k(\chi,s)=A_k(\chi,s)B_k(\chi,s),
\end{equation}
where
$$
A_k(\chi,s)
:=\sum\limits_{n\in \mathcal{A}_k}
\frac{\chi(n)d_3(kn)}{n^s}, \quad
B_k(\chi,s)
:=\sum\limits_{(n,k)=1} \frac{\chi(n)d_3(n)}{n^s}.
$$
Now we may write $A_k$ and $B_k$ as Euler products in the form
\begin{align} \label{Afact}
A_k(\chi,s)
&=\prod\limits_{p\mid k}  \sum\limits_{j=0}^{\infty} \frac{\chi(p^j)d_3\left(p^{j+v_p(k)}\right)}{p^{js}},\\
\label{Bfact}
B_k(\chi,s)&=\prod\limits_{p\mid k} \left(1-\frac{\chi(p)}{p^s}\right)^3 L^3(\chi,s).
\end{align}
Obviously, $A_k(\chi,s)$ can be analytically continued to the half
plane $\Re(s)>0$, and $B_k(\chi,s)$ can be meromorphically continued
to the whole complex plane. Moreover $B_k(\chi,s)$ is holomorphic if
$\chi$ is non-principal and has a pole at $s=1$ if $\chi$ is
principal. In the latter case, when $\chi$ is the principal character
$\chi_0$ modulo $q$, we have 
\begin{equation} \label{Bprinc}
B_k(\chi_0,s)=\prod\limits_{p\mid kq} \left(1-\frac{1}{p^s}\right)^3 \zeta^3(s).
\end{equation}
Furthermore, we have the following bounds.

\begin{lemma}  Let $k,q\in \NN$.
Let $\chi$ be a non-principal character modulo $q$. Then for $\Re(s)>1/2$ we have
\begin{equation} \label{Fnonprincbound}
\left| F_k(\chi,s)\right| \ll k^{\varepsilon}\left| L(\chi,s)\right|^3.
\end{equation}
Let  $\chi_0$ be the principal character modulo $q$. Then
for $\Re(s)>1/2$ and $s\not=1$ we have
\begin{equation} \label{Fprincbound}
\left| F_k(\chi_0,s)\right| \ll (kq)^{\varepsilon}\left| \zeta(s)\right|^3.
\end{equation}
For $j\in \{0,1, 2\}$ we have 
\begin{equation} \label{resbound}
\frac{\dif^j}{\dif^j x}\mbox{\rm Res}_{s=1} F_k(\chi_0,s)\cdot \frac{x^s}{s} \ll \frac{(qkx)^{\varepsilon}x}{x^j}.
\end{equation}
\end{lemma}

\begin{proof}
We first deduce from \eqref{Afact} that
$$
 A_k(\chi,s) \ll
 \prod\limits_{p\mid k}  p^{\nu_p(k) \varepsilon}
 \sum\limits_{j=0}^{\infty} \frac{p^{j\varepsilon}}{p^{j/2}} =
 \prod\limits_{p\mid k} \frac{p^{\nu_p(k)
 \varepsilon}}{1-p^{-1/2+\varepsilon}} 
\ll k^{\varepsilon} , 
$$
provided $\varepsilon \leq 1/4$.   Moreover, if $\chi$ is a Dirichlet
character modulo $q$ and $\Re(s) > 1/2$, we have 
$$
 \prod_{p\mid k}  \left(1-\frac{\chi(p)}{p^s}\right)^3 \ll \prod_{p\mid k}  \left(1+\frac{1}{\sqrt{2}}\right)^3 \ll k^{\varepsilon}.
$$
 Similarly, if $\Re(s) > 1/2$, then 
$$
 \prod_{p\mid kq} \left(1-\frac{1}{p^s}\right)^3 \ll (kq)^{\varepsilon}.
$$
Combining these estimates with
\eqref{factorization}, \eqref{Bfact} and \eqref{Bprinc}, we
arrive at the first pair of estimates in the statement of the lemma.

 Let $x>0$. To prove \eqref{resbound}, we note that
\[ x^s = x \sum_{n=0}^2 \frac{\log^n x}{n!} (s-1)^n + (s-1)^3 R_x(s), \]
where $R_x(s)$ is an entire function in $s$.  We have
\begin{equation} \label{differentiate}
\begin{split}
 \Res_{s=1} F_k(\chi_0,s)\cdot \frac{x^s}{s}  & =  \frac{1}{2\pi i} \ \int\limits_{|s-1|=1/3} F_k(\chi_0,s) \frac{x^s}{s} \ \dif s \\
 & = \frac{1}{2\pi i} \sum_{n=0}^2 \ \int\limits_{|s-1|=1/3} \frac{F_k(\chi_0,s)}{s} \frac{(s-1)^n}{n!} \ \dif s \cdot x \log^n x .
 \end{split}
 \end{equation}
The integral involving $(s-1)^3R_x(s)$ vanishes since $F_k(\chi_0,s)$
has triple pole at $s=1$.  We now have
$$
 \frac{\dif^j}{\dif^j x}\mbox{\rm Res}_{s=1} F_k(\chi_0,s)\cdot
 \frac{x^s}{s}  = \frac{1}{2\pi i} \sum_{n=0}^2 \
 \int\limits_{|s-1|=1/3} \frac{F_k(\chi_0,s)}{s} \frac{(s-1)^n}{n!} \
 \dif s \cdot \frac{\dif^j}{\dif^j x} x \log^n x , 
$$
for $j \in \{0,1,2\}$.  It is clear that
$$
 \frac{\dif^j}{\dif^j x} x \log^n x \ll x^{1-j+\varepsilon} .
$$
Furthermore, using \eqref{Fprincbound}, we have for $|s-1|=1/3$ and $0 \leq n \leq 2$,
$$
 \frac{F_k(\chi_0,s)}{s} \frac{(s-1)^n}{n!} \ll (qk)^{\varepsilon} \left| \zeta(s) \right|^3 \ll (qk)^{\varepsilon}.
$$
Here we have noted that 
$\zeta(s)$ is bounded above by an absolute constant for $s$ with
$|s-1|=1/3$.  Inserting these bounds into \eqref{differentiate}, we
arrive at \eqref{resbound}. 
\end{proof}

\section{Technical results}

In this section we record some of the key technical facts that will be
called upon in our method. 

\begin{lemma} \label{mikawalemma1} Let $2<\Delta<N/2$. For arbitrary 
$a_n\in \CC$, we have
$$
\int\limits_{|\beta|\le 1/\Delta} \left| \sum\limits_{N<n\le 2N} a_n e(\beta n) \right|^2\dif \beta \ll \Delta^{-2} \int\limits_{N}^{2N} \left| \sum\limits_{t<n\le t+\Delta/2} a_n\right|^2 \dif t
+\Delta \left(\sup\limits_{N<n\le 2N} \left| a_n\right|\right)^2,
$$
where the implied constant is absolute.
\end{lemma}

\begin{proof}
This is Lemma 1 in \cite{mikawa} and is a form of the
Sobolev--Gallagher inequality.
\end{proof}

The next two lemmas are modified versions of Lemmas 2 and 5 in
\cite{mikawa}, respectively, 
where the role of $\Lambda(n)$ is now taken by $d_3(n)$. 

\begin{lemma} \label{mikawalemma2} Let $k,q\in \mathbb{N}$,
  $\Delta,N>1$ and let $\chi$ be a character modulo $q$.  
Set $\delta(\chi)=1$ if $\chi$ is principal and
$\delta(\chi)=0$ otherwise. Define
$$
S(k,\chi,\Delta,N)=\int\limits_{N}^{2N} \left| \sum\limits_{x<n\le x+\Delta} \chi(n)d_3(kn)  - \delta(\chi)\cdot \Res_{s=1} 
 \frac{\left((x+\Delta)^s-x^s\right)F_k(\chi_0,s)}{s} \right|^2 \dif x.
$$
Let $0<\eta<5/12$ be given. Then there exist positive $\delta$ and $\delta_3$ depending on $\eta$ such that if $k,q\le N^{\delta}$ and $N^{1/6+\eta}\le \Delta\le N^{1-\eta}$,  we have
\begin{equation} \label{Sbound}
S(k,\chi,\Delta,N)\ll \Delta^2N^{1-\delta_3}.
\end{equation}
\end{lemma}

\begin{proof}
For $k=q=1$, Ivi\'c \cite[Corollary 1]{ivic-art} proved that there exists $\delta_3>0$ depending on $\eta$ such that if
$N^{1/6+\eta}\le \Delta\le N^{1-\eta}$, we have
$$
S(1,\chi_0,\Delta,N)=\int\limits_{N}^{2N}\ \left| \sum\limits_{x<n\le x+\Delta}
d_3(n) - \Res_{s=1} \frac{((x+\Delta)^s-x^s)\zeta^3(s)}{s}\right|^2
\dif x \ll \Delta^2N^{1-\delta_3}. 
$$
This is based on a bound for the sixth moment of the Riemann zeta
function of the expected order of magnitude on the line $\Re(s)=7/12$,
which we already made use of in section \ref{s:G}.
Ivi\'c's method can be easily generalised to yield \eqref{Sbound}. The only additional inputs are the following. If $\chi$ is principal, then we use the bound \eqref{Fprincbound}. If 
$\chi$ is non-principal, then we use the bound 
\eqref{Fnonprincbound} and 
a bound for the sixth moment of $L(\chi,s)$ in place of
$\zeta(s)$. Indeed, for any given $\varepsilon>0$, we have the bound
$$
\int\limits_{-T}^T \left|L\left(\chi,\frac{7}{12}+it\right)\right|^6
\dif t \ll T(NT)^{\varepsilon}, 
$$
provided that $q\le N^{\delta}$ with $\delta>0$ small enough. The
proof of this estimate is analogous to the proof of the corresponding
result for the Riemann zeta function and involves a generalisation of
the Atkinson mean square formula for $L$-functions due to Meurman
\cite{meurman}. 
\end{proof}

For the remainder of this section we suppose that 
$\alpha\in \RR$ is given and that there exist coprime 
integers $a,q$ such that 
$|\alpha-a/q|\le q^{-2}$ and $q<\Delta<N/2.$
Our main goal in this section is a proof of the following result.

\begin{lemma} \label{mikawalemma5}
Suppose that $\Delta>N^{1/3}$ and let
$$
J=J(\alpha,\Delta):=\int\limits_{N}^{2N} \left| \sum\limits_{t<n\le
    t+\Delta} d_3(n)e(\alpha n)\right|^2 \dif t. 
$$
Then there exists $\delta_4>0$ and $F>0$ such that
$$
J\ll (\log N)^F\left(\Delta N\left(N^{1/3}+\Delta
    q^{-1/2}+(q\Delta)^{1/2}+q\right)+\Delta^2N^{1-\delta_4}+\Delta^3\right). 
$$
\end{lemma}

The proof of this lemma requires some auxiliary results, namely 
slightly modified versions of Lemmas 6, 7 and 8 in \cite{mikawa}. 
Let $f$ and $g$ be arbitrary
sequences such that $|f(n)|\le\log n$ and $|g(n)|\le
d_5(n)$. Moreover, let $U,V,C>0$ and define 
\begin{align*}
J_1&:=\int\limits_{N}^{2N} \left| \sum\limits_{\substack{t<mn\le t+\Delta\smallskip\\ U\le m\le 2U}} g(n)e(\alpha mn)\right|^2 \dif t,\\
J_2&:=\int\limits_{N}^{2N} \left| \sum\limits_{\substack{t<dl\le
      t+\Delta\smallskip\\ C\le l\le 2C}}
  \left(\sum\limits_{\substack{mn=d\smallskip\\ U\le m\le
        2U\smallskip\\ V\le n\le 2V}} g(n) \right)e(\alpha
  dl)\right|^2 \dif t, \\
J_3&:=\int\limits_{N}^{2N} \left| \sum\limits_{\substack{t<mn\le
      t+\Delta\smallskip\\ U\le m\le 2U}} f(m)g(n)e(\alpha
  mn)\right|^2 \dif t. 
 \end{align*} 
Then we have the following bounds.

\begin{lemma} \label{mikawalemma6} 
There exists $F>0$ 
such that
$$
J_1\ll (\log N)^F\left(\Delta N\left(\Delta
    q^{-1/2}+(q\Delta)^{1/2}\right)+\Delta^2(N/U)^2+\Delta^3\right). 
$$
\end{lemma}

\begin{proof} This is Lemma 6 in \cite{mikawa} with the summation condition $m\ge U$ being replaced by $U\le m\le 2U$. The proof is similar.
\end{proof} 

\begin{lemma} \label{mikawalemma7} 
There exists $\delta_4>0$ and $F>0$ 
such that
$$
J_2\ll (\log N)^F\left(\Delta N\left(\Delta q^{-1/2}+(q\Delta)^{1/2}\right)+\Delta^3\right)+\Delta^2\left(N^{1-\delta_4}+N^{7\delta_4}U^{3/2}V^4\right).
$$
\end{lemma}

\begin{proof} This is Lemma 7 in \cite{mikawa} with an extra summation condition $C\le l\le 2C$ included and the summation conditions 
$m\le U$ and $n\le V$ being replaced by $U\le m\le 2U$ and $V\le n\le 2V$. The proof is similar. 
\end{proof}

\begin{lemma} \label{mikawalemma8} If $U<\Delta$, then
there exists $F>0$ such that
$$
J_3\ll \Delta N(\log
N)^F\left(U+\frac{\Delta}{q}+\frac{\Delta}{U}+q\right). 
$$
\end{lemma}

\begin{proof} This is Lemma 8 in \cite{mikawa}.
\end{proof}

We now turn to the proof of Lemma \ref{mikawalemma5}. To prove his
corresponding result \cite[Lemma 5]{mikawa}, with $\Lambda(n)$ in place of
$d_3(n)$, 
Mikawa employed a Vaughan type decomposition of $\Lambda$ due to
Heath-Brown. Instead, we use here  
the much simpler decomposition $d_3={\bf 1}\ast {\bf 1}\ast {\bf
  1}$.

\begin{proof}[Proof of Lemma \ref{mikawalemma5}]
For $N\le n\le 3N$, we may split 
$
d_3(n) = \sum_{abc=n} 1
$
into $O\left((\log N)^3\right)$ terms of the form
$$
d_{A,B,C}(n)=\sum\limits_{\substack{A\le a\le 2A\\ B\le b\le 2B\\ C\le
    c\le 2C\\abc=n}} 1, 
$$
with $ABC=N$. Using Cauchy--Schwarz, it follows that 
$$
J\ll \sup_{\substack{A\leq B\leq C\\ ABC=N}}  (\log N)^9 
\int\limits_{N}^{2N} \left| \sum\limits_{t<n\le t+\Delta}
  d_{A,B,C}(n)e(\alpha n)\right|^2 \dif t.
$$
Our argument breaks into the following 
three cases.

{\it Case 1}: Let $N^{\delta_4}\le A\le N^{1/3}$. We may write
$$
\int\limits_{N}^{2N} \left| \sum\limits_{t<n\le t+\Delta}
  d_{A,B,C}(n)e(\alpha n)\right|^2 \dif t =  
\int\limits_{N}^{2N} \left| \sum\limits_{\substack{t<mn\le t+\Delta\\
      A\le m\le 2A}} h_{B,C}(n)e(\alpha mn)\right|^2 \dif t, 
$$
with
$$
h_{B,C}(n):=\sum\limits_{\substack{B\le b\le 2B\\ C\le c\le 2C\\ bc=n}} 1.
$$
Now Lemma \ref{mikawalemma8} with $f= 1$ and $g=h_{B,C}$ yields the
existence of $F>0$ such that
\begin{align*}
\int\limits_{N}^{2N} \left| \sum\limits_{t<n\le t+\Delta}
  d_{A,B,C}(n)e(\alpha n)\right|^2 \dif t &\ll \Delta N (\log
N)^F\left(A+\frac{\Delta}{q}+\frac{\Delta}{A}+q\right)\\ 
&\ll \Delta N(\log N)^F\left(N^{1/3}+\Delta q^{-1}+\Delta
  N^{-\delta_4}+q\right). 
\end{align*} 

{\it Case 2}: Let $A\le N^{\delta_4}$ and $C\ge N^{1/2+\delta_4/2}$.  We have
\begin{equation} \label{transform}
\int\limits_{N}^{2N} \left| \sum\limits_{t<n\le t+\Delta} d_{A,B,C}(n)e(\alpha n)\right|^2 \dif t = 
\int\limits_{N}^{2N} \left| \sum\limits_{\substack{t<mn\le t+\Delta\\ C\le m \le 2C}} h_{A,B}(n)e(\alpha mn)\right|^2 \dif t.
\end{equation}
Now Lemma \ref{mikawalemma6} with $g=h_{A,B}$ yields 
the existence of $F>0$ such that
\begin{align*} 
\int\limits_{N}^{2N} \left| \sum\limits_{t<n\le t+\Delta}
  d_{A,B,C}(n)e(\alpha n)\right|^2 \dif t 
&\ll (\log N)^F\left(\Delta N\left(\Delta q^{-1/2}+(q\Delta)^{1/2}\right)+\Delta^2\left(N/C\right)^2+\Delta^3\right) \\
&\ll (\log N)^F\left(\Delta N\left(\Delta q^{-1/2}+(q\Delta)^{1/2}\right)+\Delta^2N^{1-\delta_4}+\Delta^3\right).
\end{align*} 

{\it Case 3}: Let $A\le N^{\delta_4}$ and $N^{1/2-3\delta_4/2}\le B
\le C\le N^{1/2+\delta_4/2}$.  By \eqref{transform} and the definition
of $h_{A,B}$, we have 
$$
\int\limits_{N}^{2N} \left| \sum\limits_{t<n\le t+\Delta} d_{A,B,C}(n)e(\alpha n)\right|^2 \dif t = 
\int\limits_{N}^{2N} \left| \sum\limits_{\substack{t<mn\le t+\Delta\\
      C\le m \le 2C}}  \left(\sum\limits_{\substack{A\le u\le 2A\\
        B\le v\le 2B\\ uv=n}} 1\right) e(\alpha mn)\right|^2 \dif t. 
$$
Lemma \ref{mikawalemma7} with $g= 1$ yields the existence of $F>0$
such that
\begin{align*} 
\int\limits_{N}^{2N} 
&\left| \sum\limits_{t<n\le t+\Delta} d_{A,B,C}(n)e(\alpha n)\right|^2 \dif t\\ 
&\ll (\log N)^F\left(\Delta N\left(\Delta q^{-1/2}+(q\Delta)^{1/2}\right)+\Delta^3\right)+\Delta^2\left(N^{1-\delta_4}+N^{7\delta_4}A^4B^{3/2}\right) \\
&\ll (\log N)^F\left(\Delta N\left(\Delta q^{-1/2}+(q\Delta)^{1/2}\right)+\Delta^3\right)+\Delta^2N^{1-\delta_4},
\end{align*} 
provided that $\delta_4<1/51$.

Since $ABC=N$ and $A\le B\le C$, there are no remaining
cases. Combining everything therefore leads to the statement of 
Lemma \ref{mikawalemma5}.  
\end{proof}

Throughout the sequel, let $\chi_{0,n}$ be the principal character
modulo $n$. In our treatment of the major arcs, we will have to
approximate the term  
$$
T(q,x,\Delta):= \sum_{k\mid q} \frac{\mu(q^*)}{\varphi(q^*)} \sum_{x/k<m \le (x+\Delta)/k} \chi_{0,q^*} (m) d_3(mk),
$$
with $q^*=q/k$, by a simpler term of the form 
$$
\sum_{x<m \le x+\Delta} p_{q}(m),
$$
where $p_q(m)$ is a certain nicely behaved 
function. The remainder of this
section is devoted to the computation of this function. 

Using Lemma \ref{mikawalemma2} we shall aim to approximate
$T(q,x,\Delta)$ in mean square by 
\begin{equation} \label{T0}
T_0(q,x,\Delta):=\sum_{k\mid q} \frac{\mu(q^*)}{\varphi(q^*)} \Res_{s=1} \frac{\left((x+\Delta)/k)^s-(x/k)^s\right) F_{k,q^*}(s)}{s},
\end{equation}
where 
\begin{equation} \label{Fkqdef}
F_{k,q^*}(s)=F_k(\chi_{0,q^*},s),
\end{equation} 
in the notation of \eqref{eq:series}.
Let
\begin{equation} \label{pkqdef}
p_{k,q^*}(x):=\frac{\dif}{\dif x} \Res_{s=1} \frac{x^sF_{k,q^*}(s)}{s}.
\end{equation}
Then
$$
\Res_{s=1} \frac{\left((x+\Delta)/k)^s-(x/k)^s\right) F_{k,q^*}(s)}{s}
=\frac{1}{k} \int\limits_{x}^{x+\Delta} p_{k,q^*}\left(\frac{t}{k}\right)\dif t.
$$
Hence we may write
$$
T_0(q,x,\Delta)=\int\limits_{x}^{x+\Delta} p_q\left(t\right) \dif t,
$$
where
\begin{equation} \label{pqdef}
p_q(t):=\sum\limits_{k\mid q} \frac{\mu(q^*)}{\varphi(q^*)k}\cdot
p_{k,q^*}\left(\frac{t}{k}\right). 
\end{equation}

From \eqref{resbound}, it follows that 
\begin{equation} \label{pkqbound}
p_{k,q^*}(x) \ll (kq^*x)^{\varepsilon}, \quad 
p_{k,q^*}'(x) \ll \frac{(kq^*x)^{\varepsilon}}{x}.
\end{equation}
This together with 
\begin{equation} \label{phibound}
\varphi(q^*)\gg \frac{q^*}{\log\log 10q^*}
\end{equation}
implies that
\begin{equation} \label{pqbound}
p_q(n)\ll \frac{(qn)^{\varepsilon}}{q}.
\end{equation} 
Armed with these we 
may approximate the above integral by a sum. 
For $N\ll x < x+\Delta \ll N$
we see that 
\begin{equation} \label{T0pqrelation}
T_0(q,x,\Delta)=\sum\limits_{x<n\le x+\Delta} p_q\left(n\right)
+O\left(\frac{(qN)^{\varepsilon}}{q}\right). 
\end{equation}

\section{Treatment of the major arcs} 
Now we investigate the major arcs. Let $\alpha\in I_{q,a}$ and write $\alpha=a/q+\beta$. Then we have
$$
S(\alpha)=\sum\limits_{N<n\le 2N} d_3(n) \cdot e\left(\frac{an}{q}\right) \cdot e(\beta n).  
$$
Breaking the sum according to the value of $(n,q)$, we obtain
$$
 S( \alpha) = \sum_{k\mid q} \sum_{\substack{N<n \le 2N \\ (n,q)=k}}
 d_3(n) e \left( \frac{an}{q} \right) e(\beta n) 
= \sum_{k\mid q} \sum_{\substack{N/k <m \le 2N/k \\ (m,q^*)=1}}
d_3(mk) e \left( \frac{am}{q^*} \right) e \left( \beta mk \right),
$$
where $q=q^* k$.  Let $\tau(\chi)$ denote the Gauss sum associated to
a Dirichlet character. Then for $(a,r)=1$ we
have the familiar identity
\[ e \left( \frac{a}{r} \right) = \frac{1}{\varphi(r)} \sum_{\chi
  \bmod{r}} \chi (a) \tau(\overline{\chi}), 
\] 
relating additive to multiplicative characters (see, for 
example, \cite[Eq.\ (3.11)]{HIEK}).
Applying this we may write
\[ S(\alpha) = \sum_{k\mid q} \frac{1}{\varphi(q^*)} \sum_{\chi \bmod{q^*}} \tau(\overline{\chi}) \chi (a) \sum_{N/k< m \le 2N/k} \chi (m) d_3(mk) e \left( \beta mk \right). \]

We write $ S(\alpha) = a + b+c$,  where
 \begin{align} \label{adef}
 a &:= \sum_{N<m \le 2N} p_q (m) e \left( \beta m\right),\\
\label{bdef}
 b &:= \sum_{k\mid q} \frac{1}{\varphi(q^*)} \sum_{\substack{\chi \bmod{q^*}\\ \chi\not=\chi_{0,q^*}}} \tau(\overline{\chi}) \chi (a) \sum_{N/k<m \le 2N/k} \chi (m) d_3(mk) e \left( \beta mk \right),
 \end{align}
 and
 \begin{equation} \label{cdef}
 c:= \sum_{k\mid q} \frac{\mu(q^*)}{\varphi(q^*)} \sum_{N/k<m \le 2N/k} \chi_{0,q^*}(m) d_3(mk) e \left( \beta mk \right)-\sum_{N<m \le 2N} p_q (m) e \left( \beta m\right).
 \end{equation}
Furthermore, set
$$
\int\limits_{\majorarc} |a|^2 \dif \alpha = A^2, \quad \int\limits_{\majorarc} |b|^2 \dif \alpha = B^2, \quad  \int\limits_{\majorarc} |c|^2 \dif \alpha = C^2.
$$
Using Cauchy--Schwarz, we get 
\begin{equation} \label{majarc}
\int\limits_{\majorarc} \left| S(\alpha) \right|^2 e ( -h \alpha) \dif
\alpha = \int\limits_{\majorarc} \left| a \right|^2 e ( -h \alpha)
\dif \alpha + 
O\left( A(B+C) + B^2+C^2\right). 
\end{equation}
To estimate the error term in \eqref{majarc}, we need bounds for $A$, $B$ and $C$ which are provided by the following lemmas.

\begin{lemma} \label{Asize}
Let $\ve>0$. Then we have
$A^2 \ll N^{1+\varepsilon}$.
\end{lemma}

\begin{proof}
Expanding $|a|^2$ and integrating, we obtain
$$
 A^2 \ll  \sum_{q \leq Q_1} \frac{1}{qQ} \sum_{\substack{1 \leq a \leq
     q \\ (a,q) =1}} \sum_{N < m \leq 2N} p^2_q(m) + \sum_{q \leq Q_1}
 \sum_{\substack{1 \leq a \leq q \\ (a,q) =1}} \mathop{\sum_{N < m_1
     \leq 2N} \sum_{N < m_2 \leq 2N}}_{m_1 \neq m_2} \left|
   \frac{p_q(m_1) p_q(m_2)}{m_1-m_2} \right|. 
$$
Now inserting the estimate \eqref{pqbound}, we easily 
arrive at our desired result.
\end{proof}

\begin{lemma} \label{Bsize}
Let $\delta>0$ be sufficiently small. Then there exists $\delta_5 >0$ depending on $\delta$ such that
$B^2 \ll N^{1-\delta_5}.$
\end{lemma}

\begin{proof}
By the definition of the major arcs, we have
\[ B^2 = \sum_{q \leq Q_1} \sum_{\substack{1 \leq a \leq q \\ (a,q) =1}}\ \int\limits_{|\beta| \leq 1/(qQ)} \left | \sum_{k\mid q} \frac{1}{\varphi(q^*)} \sum_{\substack{\chi \bmod{q^*}\\ \chi\not=\chi_{0,q^*}}} \tau(\overline{\chi}) \chi (a) \sum_{N/k<m \le 2N/k} \chi (m) d_3(mk) e \left( \beta mk \right) \right|^2 \dif \beta. \]
Writing
\[
\frac{1}{\phi(q^*)} = \sqrt{\frac{q^*}{\phi(q^*)}} \cdot
\frac{1}{\sqrt{\phi(q^*) q^*}}
\] 
and applying Cauchy--Schwarz twice, we obtain
\[ B^2 \ll \sum_{q \leq Q_1} \sum_{\substack{1 \leq a \leq q \\ (a,q) =1}}
g_q \sum_{k\mid q} \sum_{\substack{\chi \bmod{q^*}\\
    \chi\not=\chi_{0,q^*}}} \ \int\limits_{|\beta|\leq 1/(qQ)} \left|
  \sum_{N/k<m \le 2N/k} \chi (m) d_3(mk) e \left( \beta mk \right)
\right|^2 \dif \beta, \]
where
$$
g_q:=\sum_{k\mid q} \frac{q^*}{\varphi(q^*)}\ll q^\ve,
$$
using \eqref{phibound}.  Now applying Lemma \ref{mikawalemma1} with a change of variables, we get
$$
 B^2 \ll \sum_{q \leq Q_1} \sum_{\substack{1 \leq a \leq q \\ (a,q) =1}} g_q \sum_{k\mid q} \sum_{\substack{\chi \bmod{q^*}\\ \chi\not=\chi_{0,q^*}}} \left( \frac{k}{(qQ)^2}  \int\limits_{N/k}^{2N/k} \left| \sum_{x < m \leq x+qQ/(2k)}
 \chi (m) d_3(mk) \right|^2 \dif x 
+ qQ N^{\varepsilon} \right).
$$
Applying Lemma~\ref{mikawalemma2} and summing up all relevant variables, we get the
bound
$$
B^2\ll Q_1^{3+\varepsilon}N^{1-\delta_3}+Q_1^{4+\varepsilon}Q.
$$ 
This is satisfactory if $\delta<\min\{\delta_3/3,3/16\}$.
\end{proof} 

\begin{lemma} \label{Csize}
Let $\delta>0$ be sufficiently small. Then there exists $\delta_6>0$ depending on $\delta$ such that $C^2 \ll N^{1-\delta_6}$.
\end{lemma} 

\begin{proof}
First we observe that
\[ \sum_{k\mid q} \frac{\mu(q^*)}{\varphi(q^*)} \sum_{N/k<m \le 2N/k} 
\hspace{-0.2cm}
\chi_{0,q^*}(m) d_3(mk) e \left( \beta mk \right) = \sum_{N < n \leq 2N} \sum_{k\mid  (n,q)} \frac{\mu(q^*)}{\varphi(q^*)} \chi_{0,q^*} \left( \frac{n}{k} \right) d_3(n) e ( \beta n). \]
Therefore, inserting the above into \eqref{cdef}, we get that
$$
c = \sum_{N < n \leq 2N} \left( a_n d_3(n) - p_q(n) \right) e (\beta n),
$$
where
\begin{equation} \label{acoeffdef}
a_n = \sum_{k\mid  (n,q)} \frac{\mu(q^*)}{\varphi(q^*)} \chi_{0,q^*} \left( \frac{n}{k} \right) .
\end{equation}
Hence we have
\begin{equation} \label{Csqalt}
 C^2 = \sum_{q \leq Q_1} \sum_{\substack{1 \leq a \leq q \\ (a,q)=1}} I(q,a),
 \end{equation}
with 
$$
I(q,a):=\int\limits_{|\beta| < 1/(qQ)} \left| \sum_{N < n \leq 2N} \left( a_n d_3(n) - p_q(n) \right) e (\beta n) \right|^2 \dif \beta.
$$
Lemma~\ref{mikawalemma1} yields
\begin{equation} \label{applymikawa}
\begin{split}
I(q,a)
& \ll \frac{1}{(qQ)^2} \int\limits_N^{2N} \left| \sum_{t < n \leq t+qQ/2} \left( a_n d_3(n) - p_q(n) \right) \right|^2 \dif t + qQ \left( \sup_{N < n \leq 2N} |a_n d_3(n) - p_q(n) | \right)^2 \\
& \ll \frac{1}{(qQ)^2} \int\limits_N^{2N} \left| \sum_{t < n \leq t+qQ/2} \left( a_n d_3(n) - p_q(n)  \right) \right|^2 \dif t + q Q N^{\varepsilon},
\end{split}
\end{equation}
where the last estimate comes from using \eqref{pqbound} and $a_n \ll n^{\varepsilon}$.  Employing \eqref{T0pqrelation}, we have
\begin{align*}
\left| \sum_{t < n \leq t+qQ/2} (a_n d_3(n) - p_q(n)) \right|^2 & = \left| \sum_{t < n \leq t+qQ/2} a_n d_3(n) - T_0 \left( q,t, \frac{qQ}{2} \right) + O \left( \frac{(qN)^{\varepsilon}}{q} \right) \right|^2 \\
 & \ll \left| \sum_{t < n \leq t+qQ/2} a_n d_3(n) - T_0 \left( q,t,
     \frac{qQ}{2} \right) \right|^2 + O \left(
   \frac{(qN)^{\varepsilon}}{q^2} \right).
 \end{align*}
Note that $\sum_{k\mid q} \mu^2(q^*)/\varphi^2(q^*) \ll 1$.
Now using \eqref{T0}, \eqref{acoeffdef} and Cauchy--Schwarz, we deduce that the first term in the last line is
$$
\ll \sum_{k\mid q} \left| \sum_{t/k< m \leq t/k + qQ/(2k)} \chi_{0,q^*}
 (m) d_3 (mk) - \Res_{s=1} \frac{((t+qQ/2)/k)^s - (t/k)^s) F_{k,q^*}(s)}{s}
 \right|^2. 
$$
Reinserting our work back into \eqref{applymikawa}, we see after a
change of variables that
\begin{align*}
 I(q,a)\ll~&
\frac{q^{\varepsilon}N^{1+\varepsilon}}{q^{4}Q^2} + qQ N^{\varepsilon}+
\sum_{k\mid q} \frac{k}{(q Q)^2} \\
&\times 
\int\limits_{N/k}^{2N/k} \left| \sum_{x< m
    \leq x+qQ/(2k)}  \chi_{0,q^*} (m) d_3(mk) - \Res_{s=1}
  \frac{((x+qQ/(2k))^s - x^s) F_{k,q^*}(s)}{s} \right|^2 \dif t.
\end{align*}
We are now in a position to apply Lemma~\ref{mikawalemma2} to the
 integral on the right-hand side.  This gives 
$$
I(q,a)\ll 
N^{1-\delta_3} + \frac{N^{1+\varepsilon}}{q^{4}Q^2} + qQ
N^{\varepsilon} \ll N^{1-\delta_3}, 
$$
since $q\leq N^{\delta}\le N^{1/4}$ and $Q = N^{1/4}$.  Now inserting the
above estimate into \eqref{Csqalt} and summing up all the relevant
variables, we arrive at our desired result if $\delta<\delta_3/2$. 
\end{proof}

From \eqref{majarc} and Lemmas \ref{Asize}, \ref{Bsize} and
\ref{Csize}, we obtain the following result.

\begin{lemma}\label{majorlemma} Let $\delta>0$ be sufficiently
  small. Then there exists $\delta_7>0$ depending on $\delta$ such
  that uniformly for $h$, we have 
$$
\int\limits_{\majorarc} \left| S(\alpha) \right|^2 e ( -h \alpha) \dif \alpha = \int\limits_{\majorarc} \left| a \right|^2 e ( -h \alpha) \dif \alpha + O( N^{1-\delta_7}).
$$
\end{lemma}

We now turn to the computation of
\begin{equation} \label{Zhdef}
Z(h) := \int\limits_{\majorarc} |a|^2 e(-h\alpha) \dif \alpha,
\end{equation}
where $a$ is given by \eqref{adef}.
By the definition of the major arcs, we have
\begin{align*}
Z(h) & = \sum_{q \le Q_1} \sum_{\substack{1 \le a \le q \\ (a,q)=1}}
\int\limits_{|\beta| \le 1/qQ} \left| \sum_{N<m \le 2N} p_{q} (m) e \left( \beta m \right) \right|^2
e\left(-h\left(\frac{a}{q}+\beta\right)\right) \dif \beta \\
& = \sum_{q \le Q_1} c_q(-h) \int\limits_{|\beta| \le 1/qQ} \left| \sum_{N<m \le 2N} p_{q} (m) e \left( \beta m\right) \right|^2
e(-h\beta) \dif \beta,
\end{align*}
where $c_q(m)$ is the Ramanujan sum.

Expanding the square in our expression for $Z(h)$ and using \eqref{pqdef}, we have
\begin{equation} \label{Zh}
\begin{split} 
  Z(h)  
=~& \sum_{q \le Q_1} c_{q}(-h) \sum\limits_{k_1\mid q} \sum\limits_{k_2\mid q}
\frac{\mu(q/k_1)\mu(q/k_2)}{\varphi(q/k_1)k_1\varphi(q/k_2)k_2}
\\ &
\times  
  \int\limits_{|\beta| \le 1/qQ}  \sum_{N<n_1 \le 2N} \sum_{N<n_2 \le
    2N} p_{k_1,q/k_1}\left(\frac{n_1}{k_1}\right)
  p_{k_2,q/k_2}\left(\frac{n_2}{k_2}\right) 
  e(\beta(n_1-n_2-h)) \dif \beta \\
   =~& \sum_{q \le Q_1} c_{q}(-h) \sum\limits_{k_1\mid q}
  \sum\limits_{k_2\mid q}\frac{\mu(q/k_1)\mu(q/k_2)}{\varphi(q/k_1)k_1\varphi(q/k_2)k_2} \left\{ \int\limits_0^1  
    \cdots \dif \beta - \int\limits_{1/qQ}^{1-1/qQ} \cdots \dif \beta
  \right\} \\ 
   =~& \Sigma_1(h) - \Sigma_2(h), 
\end{split}
\end{equation}
say. It easily follows that
\begin{equation} \label{sigma1}
\begin{split}
\Sigma_1(h)=~&
\sum\limits_{q\le Q_1}  c_q(-h) \sum\limits_{k_1\mid q} \sum\limits_{k_2\mid q}
\frac{\mu(q/k_1)\mu(q/k_2)}{\varphi(q/k_1)k_1\varphi(q/k_2)k_2} \\
& \times 
\sum\limits_{N+h<n\le N} p_{k_1,q/k_1}\left(\frac{n}{k_1}\right) p_{k_2,q/k_2}\left(\frac{n-h}{k_2}\right).
\end{split}
\end{equation}
  
Next we turn to the estimation of 
\begin{align*} 
\Sigma_2 (h)  
=~& \sum_{q \le Q_1} c_{q}(-h) \sum\limits_{k_1\mid q} \sum\limits_{k_2\mid q}
\frac{\mu(q/k_1)\mu(q/k_2)}{\varphi(q/k_1)k_1\varphi(q/k_2)k_2}
\\ &\times 
\int\limits_{1/qQ}^{1-1/qQ}\left( \sum\limits_{N<n_1 \le 2N} p_{k_1,q/k_1}\left(\frac{n_1}{k_1}\right) e(\beta n_1) \sum\limits_{N<n_2\le 2N}
p_{k_2,q/k_2}\left(\frac{n_2}{k_2}\right) e(-\beta n_2)\right) e(-\beta h) \dif \beta.
\end{align*}
Using partial summation, \eqref{pkqbound} and the familiar bound
\begin{equation} \label{geom}
  \sum_{s < n \le t} e(\beta n) \ll \| \beta \|^{-1},
\end{equation}
where $\| \alpha \|$ is the distance of $\alpha$ to the nearest
integer, we obtain the estimate   
$$
\sum_{N<n \le 2N} p_{k,q/k}\left(\frac{n}{k}\right) e(\pm \beta n) \ll (qN)^{\varepsilon} \| \beta \|^{-1}.
$$
Since $|c_q(-h)|\le\varphi(q)$, it follows that
 $
 \Sigma_2 (h)\ll N^{\varepsilon}Q_1Q\ll N^{3/4},
$
since $\delta< 1/4$.
Combining this with \eqref{Zh}, we obtain
\begin{equation} \label{Zhapp}
Z(h)=\Sigma_1(h)+O(N^{3/4}),
\end{equation}
uniformly for $h\in \mathbb{N}$.

\section{Computation of the singular series}\label{s:SS}

We now show that our main term $\Sigma_1(h)$ in \eqref{sigma1}
can be approximated by the integral on the right-hand side
of the estimate in Proposition \ref{majorarcstheo}. Throughout this
section, we assume that 
$q\leq N^{\delta}$ and $k_i\mid q$ for $i=1,2$, 
and that $0<\delta<1/4$ and $0<\eta<1$.
In the following, we shall
frequently make use of \eqref{pkqbound}, \eqref{phibound} and
the inequality 
$|c_q(-h)|\le (q,h)$ without further mention.

The innermost sum on the right-hand side of \eqref{sigma1} is
\begin{align*}
\sum\limits_{N+h<n\le 2N} 
&p_{k_1,q/k_1}\left(\frac{n}{k_1}\right)
p_{k_2,q/k_2}\left(\frac{n-h}{k_2}\right)\\ 
&= \sum\limits_{N<n\le 2N} p_{k_1,q/k_1}\left(\frac{n}{k_1}\right) p_{k_2,q/k_2}\left(\frac{n-h}{k_2}\right)+O(hN^{\varepsilon})\\
&= \sum\limits_{N<n\le 2N} p_{k_1,q/k_1}\left(\frac{n}{k_1}\right)
p_{k_2,q/k_2}\left(\frac{n}{k_2}\right)+O(hN^{\varepsilon})\\ 
&= \int\limits_{N}^{2N} p_{k_1,q/k_1}\left(\frac{x}{k_1}\right) p_{k_2,q/k_2}\left(\frac{x}{k_2}\right) \dif x+O(hN^{\varepsilon}).
\end{align*}
It follows that
$$
\Sigma_1(h) = \sum\limits_{q\le Q_1} \frac{c_q(-h)}{q^2} \cdot \int\limits_{N}^{2N} \left(\sum\limits_{k\mid q} \frac{\mu(q/k)q}{\varphi(q/k)k} \cdot p_{k,q/k}\left(\frac{x}{k}\right)\right)^2 \dif x+ O\left(N^{\varepsilon}\sum\limits_{q=1}^{\infty} \frac{h\cdot (q,h)}{q^2}\right). 
$$
We note that uniformly for $h\le N^{1-\eta}$, we have
$$
N^{\varepsilon}\sum\limits_{q=1}^{\infty} \frac{h\cdot (q,h)}{q^2} \ll
N^{1-\delta_{8}}
$$
for some $\delta_8>0$ depending on $\eta$, if 
$2\varepsilon<\eta$. Moreover, we can extend 
to infinity the sum over $q\le Q_1$ in the main term, with
acceptable error depending on $\delta$ and $\eta$.
Combining everything, we obtain
$$
\Sigma_1(h)=\int\limits_{N}^{2N} \mathfrak{S}^*(x,h)\dif x +
O(N^{1-\delta_{9}}), 
$$
where $\delta_9$ depends on $\eta$ and $\delta$, and
$$
\mathfrak{S}^*(x,h):=\sum\limits_{q=1}^{\infty} \frac{c_q(-h)}{q^2}
\cdot \left(\sum\limits_{k\mid q} \frac{\mu(q/k)q}{\varphi(q/k)k} \cdot
  p_{k,q/k}\left(\frac{x}{k}\right)\right)^2. 
$$

We proceed to show that 
\begin{equation} \label{singularequal}
\mathfrak{S}^*(x,h)=\mathfrak{S}(x,h), 
\end{equation}
where the right-hand side is defined as in \eqref{eq:dijon}.
To begin with we write
$$
\mathfrak{S}^*(x,h)=\sum\limits_{q=1}^{\infty} \frac{c_q(-h)}{q^2}
\cdot P^*(x,q)^2, 
$$
where 
\begin{equation} \label{P3}
P^*(x,q):=\sum\limits_{d\mid q} \frac{\mu(d)d}{\varphi(d)} \cdot
p_{q/d,d}\left(\frac{xd}{q}\right). 
\end{equation}
In particular, it follows from \eqref{pkqbound}
that 
\begin{equation}
  \label{eq:P3-upper}
P^*(x,q)\ll (qx)^\ve,
\end{equation}
which is not of importance in the rest of this section but was used in section \ref{s:F}.
Recalling the definition of $p_{k,q^*}$ from \eqref{pkqdef}, we have
\begin{align*}
p_{q/d,d}(y)
&=\frac{\dif}{\dif t} \left. \Res_{s=1} \frac{t^s F_{q/d,d}(s)}{s}
\right|_y\\
&= \Res_{s=1} \left( \left.\left(\frac{\dif}{\dif t} \frac{t^s}{s}\right)\right|_y \cdot F_{q/d,d}(s)\right)\\
&=  \Res_{s=1}\ y^{s-1} \cdot F_{q/d,d}(s). 
\end{align*}
Making the change of variables $s\rightarrow s+1$, we obtain
\begin{align*}
P^*(x,q) &= \sum\limits_{d\mid q} \frac{\mu(d)d}{\varphi(d)} \cdot \Res_{s=1} \left(\frac{xd}{q}\right)^{s-1} \cdot F_{q/d,d}(s)\\
&= \Res_{s=0} \sum\limits_{d\mid q} \frac{\mu(d)d^{s+1}x^s}{\varphi(d)q^s} \cdot F_{q/d,d}(s+1). 
\end{align*}
Hence
$$
P^*(x,q)=\Res_{s=0}\ \zeta^3(s+1)H^*(s+1,q)\cdot \left(\frac{x}{q}\right)^s,
$$
where
$$
H^*(s,q):=\sum\limits_{d\mid q} \frac{\mu(d)}{\varphi(d)} \cdot d^s\cdot  G^*_{q/d,d}(s)
$$
and 
$$
G^*_{q/d,d}(s):=\frac{F_{q/d,d}(s)}{\zeta^3(s)}.
$$

For the proof of \eqref{singularequal}, it remains to show that 
$
G^*_{q/d,d}(s)=G_{q/d,d}(s),
$
in the notation of \eqref{Gkddef}. It suffices
to check 
this equation for prime powers $q=p^\alpha$, for $\alpha\in \NN$.
We recall \eqref{factorization}, \eqref{Bprinc} and \eqref{Fkqdef}.

{\it Case 1}: If $d=1$, then
$$
G_{q/d,d}^*(s) = G_{q,1}^*(s)
=  \frac{F_{p^{\alpha},1}(s)}{\zeta^3(s)}
= \left(1-p^{-s}\right)^3 \sum\limits_{j=0}^{\infty} \frac{d_3\left(p^{j+\alpha}\right)}{p^{js}} = G_{q,1}(s) = G_{q/d,d}(s).
$$

{\it Case 2}: If $d=p^{\alpha}$, then
$$
G^*_{q/d,d}(s)=G^*_{1,q}(s)=\frac{F_{1,p^{\alpha}}(s)}{\zeta^3(s)}=\left(1-p^{-s}\right)^3=G_{1,q}(s)=G_{q/d,d}(s).  
$$

{\it Case 3}: If $d=p^{\beta}$ with $1\le \beta\le \alpha-1$, then
$$
G^*_{q/d,d}(s) =
\frac{F_{p^{\alpha-\beta},p^{\beta}}(s)}{\zeta^3(s)}=\left(1-p^{-s}\right)^3\cdot
d_3(p^{\alpha-\beta})= G_{q/d,d}(s). 
$$

In this way we see that $G_{q/d,d}(s)$ and $G^*_{q/d,d}(s)$ match up in all
cases. Combining the facts in this section, we obtain the following
estimate. 

\begin{lemma} \label{Sigma1lemma} 
There exists $\delta_{10}>0$ depending on $\eta$ and $\delta$ such
that, uniformly for $h\le N^{1-\eta}$, we have 
$$
\Sigma_1(h)=\int\limits_{N}^{2N} \mathfrak{S}(x,h) \dif x +
O(N^{1-\delta_{10}}). 
$$
\end{lemma}

Combining Lemma \ref{majorlemma}, \eqref{Zhdef}, \eqref{Zhapp} and
Lemma \ref{Sigma1lemma} proves Proposition \ref{majorarcstheo}.

\section{Treatment of the minor arcs}
This last section is concerned with the proof of Proposition
\ref{minorarcstheo}, following precisely Mikawa's treatment. Expanding
the square,  re-arranging the order of summation and integration, and
using the  bound \eqref{geom}, we have 
\begin{equation} \label{homer}
\sum\limits_{h\le H} \left| \int\limits_{\mathfrak{m}} |S(\alpha)|^2 e(-\alpha h) \dif \alpha \right|^2 \ll \int\limits_{\mathfrak{m}}\int\limits_{\mathfrak{m}} |S(\alpha_1)|^2|S(\alpha_2)|^2
\min\left(H,\frac{1}{\| \alpha_1-\alpha_2 \|}\right)\dif \alpha_1\dif\alpha_2.
\end{equation}
Set $\Delta:=HN^{-\delta_{11}}$ with $0<\delta_{11}<\eta$.  We split
the right-hand side of \eqref{homer} into $I_1+I_2$, with
\begin{align*}
I_1
&:=\int\limits_{\mathfrak{m}} \int\limits_{\substack{\mathfrak{m}\\ |\alpha_2-\alpha_1|> 1/\Delta}} |S(\alpha_1)|^2|S(\alpha_2)|^2
\min\left(H,\frac{1}{\| \alpha_1-\alpha_2 \|}\right)\dif \alpha_2\dif \alpha_1,\\
I_2&:=\int\limits_{\mathfrak{m}} \int\limits_{\substack{\mathfrak{m}\\ |\alpha_2-\alpha_1|\le 1/\Delta}} |S(\alpha_1)|^2|S(\alpha_2)|^2
\min\left(H,\frac{1}{\| \alpha_1-\alpha_2 \|}\right) \dif \alpha_2\dif \alpha_1.
\end{align*}

Using orthogonality and the estimate $d_3(n)\ll n^{\varepsilon}$, we
see that
\begin{equation} \label{I1bound}
I_1\ll HN^{-\delta_{11}}\left(\int\limits_0^1 |S(\alpha)|^2 \dif \alpha\right)^2 \ll HN^{2-\delta_{11}/2}.
\end{equation}
Furthermore, we have
\begin{equation} \label{I2bound}
I_2\ll H\int\limits_{\mathfrak{m}} |S(\alpha)|^2 \left( \ \int\limits_{|\beta|\le 1/\Delta} |S(\alpha+\beta)|^2\dif \beta\right)\dif \alpha.
\end{equation}
In view of Lemma \ref{mikawalemma1},
the inner integral here is 
$$
\int\limits_{|\beta|\le 1/\Delta} |S(\alpha+\beta)|^2\dif \beta \ll
\Delta^{-2}\int\limits_{N}^{2N} \left|\sum\limits_{t<n\le t+\Delta/2}
  d_3(n)e(\alpha n)\right|^2 \dif t + \Delta N^{\varepsilon}. 
$$
Now, by Dirichlet's theorem and the definition of the minor arcs, if $\alpha\in \mathfrak{m}$, there exist $a$ and $q$ such that
$$
\left| \alpha-\frac{a}{q}\right| \le q^{-2}, \quad (a,q)=1, \quad Q_1<q\le Q.
$$
From Lemma \ref{mikawalemma5}, the definitions of $\Delta$, $Q_1$, $Q$
and the assumption $N^{1/3+\eta}\le H\le N^{1-\eta}$ it
now follows, uniformly for $\alpha\in \mathfrak{m}$, that 
$$
\Delta^{-2}\int\limits_{N}^{2N} \left|\sum\limits_{t<n\le t+\Delta/2} d_3(n)e(\alpha n)\right|^2 \dif t\ll N^{1-\delta_{12}},
$$
provided that $\delta_{12}<\min\{\delta/2,\delta_4,\eta-\delta_{11},1/24\}$.
Combining this with \eqref{I2bound}, we therefore obtain
$$
I_2\ll HN^{1-\delta_{12}} \int\limits_0^1 |S(\alpha)|^2 \dif \alpha
\ll  HN^{2-\delta_{12}/2}. 
$$
Proposition \ref{minorarcstheo} now follows on inserting this estimate
into \eqref{homer}, together with \eqref{I1bound}.

\end{document}